\documentclass[a4paper,12pt,twoside]{article}

\usepackage{amssymb,amsmath,amsthm,latexsym}
\usepackage{amsfonts}
\usepackage{graphicx,caption,subcaption}
\usepackage[pdftex,bookmarks,colorlinks=false]{hyperref}
\usepackage[hmargin=1.2in,vmargin=1.2in]{geometry}
\usepackage[mathscr]{euscript}

\usepackage[auth-lg,affil-sl]{authblk}
\allowdisplaybreaks

\usepackage{float}
\usepackage{calc,tikz}
\usetikzlibrary{arrows,decorations.markings,positioning,shapes.geometric,}
\tikzstyle{vertex}=[circle, draw, inner sep=1pt, minimum size=4pt]

\tikzstyle{ann} = [fill=white,font=\footnotesize,inner sep=1pt]
\tikzstyle{arrow} = [thick,<-->,>=stealth]

\usepackage{multirow}
\usepackage{hhline}
\usepackage{booktabs}
\usepackage{longtable}

\usepackage{url}

\newcommand{\noi}{\noindent}

\newtheorem{theorem}{Theorem}[section]
\newtheorem{definition}[theorem]{Definition}

\newtheorem{proposition}[theorem]{Proposition}

\newtheorem{remark}{Remark}

\newcommand{\keywordsname}{Keywords}
\newenvironment{keywords}{
	\normalfont
	\list{}{
		\labelwidth0pt
		\leftmargin0pt \rightmargin\leftmargin
		\listparindent\parindent \itemindent0pt
		\parsep0pt
		}
	\item[\hskip\labelsep\bfseries\keywordsname:] \small}{\endlist}

\newcommand{\mscname}{MSC 2010}
\newenvironment{msc}{
	\normalfont
	\list{}{
		\labelwidth0pt
		\leftmargin0pt \rightmargin\leftmargin
		\listparindent\parindent \itemindent0pt
		\parsep0pt
		}
	\item[\hskip\labelsep\bfseries\mscname:] \small}{\endlist}

\allowdisplaybreaks

\title{\sc $\delta^{(k)}$-Colouring of Cycle Related Graphs}

\author{\sc Johan Kok}

\affil{\small Centre for Studies in Discrete Mathematics\\Vidya Academy of Science \& Technology\\Thrissur, Kerala, India.\\ Email: {\tt kokkiek2@wtshwane.gov.za}}

\author{\sc N.K. Sudev}

\affil{\small Department of Mathematics\\ CHRIST (Deemed to be University)\\ Bengaluru-560029, INDIA.\\ Email: {\tt sudevnk@gmail.com}}

\date{}

\begin{document}
\maketitle


\hrule

\begin{abstract}
\noindent With respect to a proper colouring of a graph $G$, we know that $\delta(G) \leq \chi(G) \leq \Delta(G)+1$. If distinct colours represent distinct technology types to be located at vertices the question arises on how to place at least one of each of $k$, $1\leq k < \chi(G)$ technology types together with the minimum adjacency between similar technology types. In an improper colouring an edge $uv$ such that $c(u)=c(v)$ is called a bad edge. In this paper, we introduce the notion of $\delta^{(k)}$-colouring which is a near proper colouring of $G$ with exactly $1\leq k < \chi(G)$ distinct colours which minimizes the number of bad edges.
\vspace{0.1cm}

\vspace{0.1cm}

\end{abstract}
\vspace{0.2cm}

\hrule

\section{Introduction}

For general notation and concepts in graphs and digraphs see \cite{1,2,4}. Unless stated otherwise, all graphs we consider in this paper are non-empty, finite, simple and connected graphs.

A \textit{vertex colouring} of a graph $G$ is an assignment $c:V(G) \mapsto \mathcal{C}$, where $\mathcal{C}= \{c_1,c_2,c_3,\dots,c_\ell\}$ is a set of distinct colours. A vertex colouring is said to be a \textit{proper vertex colouring} of $G$ if no two distinct adjacent vertices have the same colour. The \textit{chromatic number} of $G$, denoted $\chi(G)$, is the minimum cardinality of a proper vertex colouring of $G$. We call such a colouring a $\chi$-colouring or a \textit{chromatic colouring} of $G$. A chromatic colouring of $G$ is denoted by $\varphi_\chi(G)$. Generally, a graph $G$ of order $n$ is said to be \textit{$k$-colourable} if $\chi(G) \leq k$. Unless mentioned otherwise, a set of colours will mean a set of distinct colours. Generally, the set $c(V(G))$ is a subset of $\mathcal{C}$. A set $\{c_i \in \mathcal{C}: c(v)=c_i$ for at least one $v\in V(G)\}$ is called a \textit{colour class} of the colouring of $G$. If $\mathcal{C}$ is a minimal proper colouring of $G$, we use the term $c(G)$ in this discussion instead of $c(V(G))$. Hence, $c(G) \Rightarrow \mathcal{C}$ and $|c(G)| = |\mathcal{C}|$. The colouring of a subgraph induced by $X\subseteq V(G)$ is denoted by $c(\langle X\rangle)$. Indexing the vertices of a graph $G$ by $v_1,v_2,v_3,\ldots,v_n$ or written as $v_i,\ i = 1,2,3,\ldots,n$, is called minimum parameter indexing. In a similar manner, a \textit{minimum parameter colouring} of a graph $G$ is a proper colouring of $G$ which consists of the colours $c_i;\ 1\le i\le \ell$, where $\ell$ is the chromatic number of $G$. 

\section{$\delta^{(k)}$-Colouring of Graphs}

If the available number of distinct colours is less than the chromatic number of $G$, then any colouring will be an improper colouring  and this will result in the formation of edges with end vertices of the the same colour. An edge $uv$ such that $c(u)=c(v)$ is called a \textit{bad edge}. 
 
If distinct colours represent distinct technology types to be located at vertices, then the question on the possibilities of placing at least one of each of technology types on the vertices of a graph subject to the condition that minimum adjacency between similar technology types occurs. The answer to the question is to maximize properness of the colouring. Properness is measured by certain colouring rules as per which, at most certain number of colour class are permitted to have adjacency among their elements and hence the objective of these rules is to minimize the number of bad edges resulting from an improper colouring. Hence, we can define the following notions:

\begin{definition}{\rm 
A \textit{near proper colouring} is a colouring which minimises the number of bad edges by restricting the number of colour classes that can have adjacency among their own elements.
}\end{definition}

\begin{definition}{\rm 
The \textit{$\delta^{(k)}$-colouring} is a near proper colouring of $G$ consisting of $k$ given colours, which minimizes the number of bad edges by permitting at most one colour class to have adjacency among the vertices in it. 
}\end{definition}

\noi The number of bad edges which result from a $\delta^{(k)}$-colouring of $G$ is denoted by, $b_k(G)$. 

The motivation for the number range of colours is that $\delta^{(k)}$-colouring has a relation with $k$-defect colouring (see \cite{3}). A trivial observation is that for a colouring $\mathcal{C}=\{c_1\}$ or for a colouring $\mathcal{C}'=\{c_i:1\leq i\leq k, k\geq 2\}$ and any colour $c_j\in \mathcal{C}'$, a defect colouring of graph $G$ such that all edges are bad is easily possible through a $\delta^{(1)}$-colouring. Also, if $c_t \in \mathcal{C}',\ c_t\neq c_1,\ c_t\neq c_j$ is introduced to colour a vertex $v\in V(G)$ in the $\delta^{(1)}$-colouring, then $d_G(v)$ edges turn good in the resulting $\delta^{(2)}$-colouring. Clearly, a $\delta^{(k)}$-colouring, $1\leq k < \chi(G)$ exists. It follows easily that for all path graphs $P_n,\ n\geq 2$, a $\delta$-colouring has $k=1$ and all $n-1$ edges are bad. Hence, $b_1(P_n)=n-1$. In fact the same result holds for all $2$-colourable (bipartite) graphs. Henceforth, the graphs $G$ under further study will have, $\chi(G)\geq 3$ and $k\geq 2$.

\begin{proposition}\label{Prop-2.1}
For odd $n\ge 3$, the number of bad edges in a cycle $C_n$ with respect to a $\delta^{(k)}$-colouring is $b_2(C_n)=1$. Furthermore, there are $2n$ such $\delta^{(2)}$-colourings.
\end{proposition}
\begin{proof}
Let $n$ be odd and the vertices of a cycle graph be labeled successively as $v_1,v_2,v_3,\ldots,v_n$. Since $\chi(C_n)=3$, a $\delta^{(2)}$-colouring is permissible. Begin with the specific alternating colouring, $c(v_1)=c_1$, $c(v_2)=c_2$, $c(v_3)=c_1,\ldots, c(v_{n-1})=c_2$. Clearly, $c(v_n) = c_1$ results in edge $v_1v_n$ bad and $c(v_n) = c_2$ results in edge $v_{n-1}v_n$ bad. In both cases, $b_2(C_n)=1$. 

The colouring of the vertex $v_n$ permitted two options. Similarly, the colouring of vertices $v_i,\ 1\leq i \leq n-1$ can alternate between $c_1$ and $c_2$. Therefore, there are $2n$ such $\delta^{(2)}$-colourings.
\end{proof}

In \cite{3}, the $k$-defect polynomial for cycle graphs is given as $\phi_k(C_n;\lambda) = \binom{n}{k}[(\lambda-1)^{n-k} +(-1)^{(n-k)}(\lambda-1)]$. In our context, $k$ means number of colours and $b_k(C_n)$ means number of bad edges. Hence, for $n$ is odd and with the notational interchange to $\lambda=2$ and $k=1$ we have, $\phi_1(C_n;2)=2n$. This corresponds with  Proposition \ref{Prop-2.1}.

A wheel graph is a graph obtained by attaching one edge from each of the vertices of a cycle to an external vertex. That is, $W_{n+1} = C_n+K_1,\ n\geq 3$. The cycle $C_n$ in the wheel $W_{n+1}$ is called the rim of the wheel and the edges and vertices of $C_n$ are respectively called rim edges and rim vertices.  The vertex corresponding to $K_1$ is called the central vertex, say $v_0$ and the edges incident with the central vertex are called spokes. The following theorem discusses the number of bad edges in a wheel graph, with respect to $\delta^{(k)}$-colourings.

\begin{proposition}\label{Prop-2.2}
For the wheel graph $W_{1,n},\ n\geq 3$, 
\begin{enumerate}\itemsep0mm
\item[(i)] $b_2(W_{1,n}) = \frac{n}{2}$ if $n$ is even. Furthermore, there are $4$ ways of obtaining a $\delta^{(2)}$-colouring.
\item[(ii)] $b_2(W_{1,n}) = \lceil\frac{n}{2}\rceil$ if $n$ is odd. Furthermore, there are $4n$ ways of obtaining a $\delta^{(2)}$-colouring.
\item[(iii)] $b_3(W_{1,n}) = 1$ if $n$ is odd. Furthermore, there are $3n$ ways of obtaining a $\delta^{(3)}$-colouring.
\end{enumerate}
\end{proposition}
\begin{proof}
(i) It is known that $\chi(C_n)=2$ if $n$ is even and there are $2$ ways to obtain a chromatic colouring of the rim $C_n$. For this partial colouring no bad edges result which is an absolute minimum. The central vertex $v_0$ can be coloured either $c_1$ or $c_2$. It implies that exactly $\frac{n}{2}$ edges are bad. Since the rim can be coloured chromatically in $2$ ways and the central vertex can be coloured in 2 ways the wheel graph has 4 possible $\delta^{(2)}$-colourings.

(ii) From the proof of Part 1 of Proposition \ref{Prop-2.1}, it follows that the rim can be coloured to result in a minimum of one bad edge. Clearly, if $c(v_n)=c_1$ the central vertex must be coloured, $c(v_0) =c_2$ and vice versa. In both cases exactly $\lceil \frac{n}{2}\rceil$ bad edges result as a minimum. Therefore, the result follows. Since, $2$-colourings of $v_n$ are permitted to permute with $2$ colourings of the vertices $v_i,\ 1\leq i \leq n-1$ where after the colour of $v_0$ is prescribed and since the rim colourings may rotate through the $n$ rim vertices, a total of $4n$ distinct $\delta^{(2)}$-colourings exist.

(iii) The rim vertices $v_,\ 1\leq i \leq n-1$ can be properly coloured with two colours in $\binom{3}{2}$ ways. The vertices $v_0$, $v_n$ can be coloured with the third colour to result in 1 bad edge as a minimum. Hence, $b_3(W_{1,n}) = 1$. Since $\binom{3}{2}=3$ and the rim colourings may rotate through the $n$ rim vertices, a total of $3n$ distinct $\delta^{(3)}$-colourings exist.
\end{proof}

A \textit{helm graph} $H_{1,n}$ is obtained from the wheel graph $W_{1,n}$ by adding a pendant vertex $u_i$ to each rim-vertex $v_i$, $1\leq i \leq n$. The number of bad edges in a helm graph is determined in the following theorem.

\begin{proposition}\label{Prop-2.3}
For helm graphs, $H_{1,n},\ n\geq 3$, 
\begin{enumerate}\itemsep0mm 
\item[(i)] $b_2(H_{1,n})= \frac{n}{2}$ if $n$ is even. Furthermore, there are 4 ways of obtaining a $\delta^{(2)}$-colouring.
\item[(ii)] $b_2(H_{1,n})= \lceil\frac{n}{2}\rceil$ if $n$ is odd. Furthermore, there are $4n$ ways of obtaining a $\delta^{(2)}$-colouring.
\item[(iii)] $b_3(W_{1,n}) = 1$ if $n$ is odd. Furthermore, there are $3n\cdot2^n$ ways of obtaining a $\delta^{(3)}$-colouring.
\end{enumerate}
\end{proposition}
\begin{proof}
The proof of $(i), (ii)$ and the first part of $(iii)$ follow by similar reasoning used in the proof of Proposition \ref{Prop-2.2}. The second part of $(iii)$ can be proved as follows. For the wheel subgraph of $H_{1,n}$, the result of Proposition \ref{Prop-2.2}(iii) holds. Since each pendant vertex can be coloured properly in two ways without resulting in additional bad edges, the pendant vertices can independently be coloured in $2^n$ ways for each colouring of the wheel subgraph. Therefore, the result follows.
\end{proof}

\begin{theorem}\label{Thm-2.4}
For $1\leq k < \chi(G)$, a graph $G$ will have a subgraph $H$ with maximum vertices such that $\chi(H)= k$.
\end{theorem}
\begin{proof}
Consider a $\delta^{(k)}$-colouring of a graph $G$ which always exists. Let $E' =\{e_i: e_i \mbox{is a bad edge}\}$. Clearly, $|E'|=b_k(G)$. Now, select a minimum set of vertices, say $S$, such that each bad edge is incident with a vertex in $S$. Such a set $S$ always exists. Then, $\chi(H) =\chi(\langle V(G)-S\rangle)=k$. Also, because $|S|$ is a minimum, it follows that $H\subseteq G$ and $|V(H)|$ is a maximum.
\end{proof}

In context of cycle related graphs, a complete graph $K_n,\ n \geq 4$ is a graph obtained from $C_n$ with all possible distinct chords added. For graphs with relative large chromatic numbers, it is important to recall the $\delta^{(k)}$-colouring rule (that is, the condition that at most one colour class is permitted adjacency). Hence, we have

\begin{theorem}\label{Thm-2.5}
For a complete graph $K_n$, the value $b_k(K_n) = \frac{x(x+1)}{2}$, $x = 1,2,3,\dots n-2$, $k= n-x$. Furthermore, there are $(n-x)\binom{n}{x+1}\cdot[n-(x+1)]!$ distinct colourings corresponding to $b_k(K_n)$. 
\end{theorem}
\begin{proof}
Part 1: It is known that $\chi(K_n) = n$. For $k=n-1$, the colour of $v_i,\ 1\leq i \leq n-1$ can be given by $c(v_i)=c_i$ and $c(v_n)=c_{n-1}$. The minimum number of bad edges in this case is $1$ and the only bad edge is $v_{n-1}v_n$. Hence, $b^*_{n-1}(K_n) = b_{n-1}(K_n)=1$. Also, $\frac{1(1+1)}{2} = 1$. Also, any edge $v_iv_j \cong K_2$. Since any subset of vertices $\binom{n}{t}, 2 \leq t \leq n-1$ induces a complete graph, the number of bad edges can only be $1,3,6,10,15,\ldots,\frac{(n-1)(n-2)}{2}$. It is easy to derive that the number of bad edges corresponds to the sequence, $b_k(K_n) = \frac{x(x+1)}{2}$, $x = 1,2,3,\dots n-2$, $k= n-x$.

Part 2: If $2\leq t \leq n-1$, the vertices of a complete graph $K_n$ have the same colour, say $c_i$, then the corresponding clique $K_t$ has $c(K_t)=c_i$. Hence, $\frac{t(t-1)}{2}$ bad edges exist. It follows that for $x = 1,2,3,\ldots n-2$ and $k= n-x$, exactly $\binom{n}{x+1}$ distinct cliques of order $x+1$ exist. Hence, $(n-x)\binom{n}{x+1}\cdot[n-(x+1)]!$ colourings exist each of which, permits exactly, $b_k(K_n) = \frac{x(x+1)}{2}$ bad edges.
\end{proof}

In the context of cycle related graphs, a complete graph $K_n,\ n \geq 4$ is a graph obtained from $C_n$ with all possible distinct chords added. For graphs with relative large chromatic numbers such as complete graphs it is important to recall the $\delta^{(k)}$-colouring rule (that at most one colour class is permitted adjacency).

In \cite{3}, it is claimed that at the time of writing only two explicit expressions of $k$-defect polynomials are known (for trees and cycle graphs). We add an initial result for determining $b_k(K_n)$.

\begin{theorem}\label{Thm-2.6}
For complete graphs $K_n$ and $\mathcal{C}= \{c_i:1\leq i \leq \lambda\}$ be a set of $\lambda$ distinct colours and for $x = 1,2,3,\ldots n-2$ and $k= n-x$, over all colour subsets, $\binom{\lambda}{k}$ we have, $\phi_k(K_n;\lambda) = \binom{n}{n-k+1}\lambda^{(k)}$. 
\end{theorem}
\begin{proof}
The result is a direct consequence of the proof of Theorem \ref{Thm-2.5}. That is, $\phi_k(K_n;\lambda)=\binom{\lambda}{k}\cdot (n-x)\binom{n}{x+1}\cdot[n-(x+1)]! = k!\cdot\binom{n}{n-k+1}\binom{\lambda}{k} = \binom{n}{n-k+1}\frac{\lambda!}{(\lambda-k)!} = \binom{n}{n-k+1}\lambda^{(k)}$. 
\end{proof}

\begin{remark}{\rm 
$\lambda^{(k)} = \lambda(\lambda-1)(\lambda-2)\cdots(\lambda-k+1)$, the falling factorial.

It is easy to see that for a graph $G$, $b_k(G) \geq b_{k+1}(G)$. The elementary graph operation between two graphs which does not add new edges is the disjoint union of graphs. Since $\chi(G \cup H) = \max\{\chi(G),\chi(H)\}$, assume without loss of generality that $\chi(G)\leq \chi(H)$. Let $\mathcal{C} = \{c_i:1\leq i \leq \chi(G)-1\}$, $t\leq |\mathcal{C}|$ and $\mathcal{C}'=\{c_j:1\leq j \leq \chi(H)-1\}$, $k\leq |\mathcal{C}'|$ then, $\mathcal{C}\subseteq \mathcal{C}'$. Furthermore, let $t=k$ if and only if $k\leq |\mathcal{C}|$.  It follows that, $b_k(G\cup H)\leq b_t(G)+b_k(H)$.
}\end{remark}

The elementary graph operation which results in a maximum number of new edges between a copy of both graphs $G$, $H$ is the join of graphs. It is known that, $\chi(G+H) = \chi(G)+\chi(H)$. The number of times a colour $c_i$ is allocated to vertices of a graph $G$ is denoted by $\theta_G(c_i)$. This brings the next result.

\begin{theorem}\label{Thm-2.7}
	For graphs $G$ and $H$ with $\chi(G)\leq \chi(H)$, $\mathcal{C} = \{c_i:1\leq i \leq \chi(G)-1\}$, $t\leq |\mathcal{C}|$ and $\mathcal{C}'=\{c_j:1\leq j \leq \chi(G+H)-1\}$, $k\leq |\mathcal{C}'|$ and $t=k$ if and only if $k\leq |\mathcal{C}|$, it follows that $b_k(G+H) \leq b_t(G) + b_k(H) + \sum\limits_{i=1}^{k}\theta_G(c_i)\theta_H(c_i)$.
\end{theorem}
\begin{proof}
	From the definitions of $\mathcal{C}$ and $\mathcal{C}'$, the terms $b_t(G)$ and $b_k(H)$ follow immediately. For $\theta_G(c_i)$, $\theta_H(c_i)$ the definition of the join between $G$ and $H$ implies that exactly $\theta_G(c_i)\theta_H(c_i)$ bad edges $vu$, $v\in V(G)$, $u\in V(H)$ exist in $G+H$. Since the minimum number of bad edges of $G$ and $H$ may result from different number of times each colour is allocated the products found in the sum term are not necessarily constants. Therefore, the $\leq$-inequality follows in terms of the definition of a $\delta^{(k)}$-colouring.
\end{proof}

Note that in Theorem \ref{Thm-2.7}, it is implied that $t=|\mathcal{C}|$ if and only if $k\geq |\mathcal{C}|$. If the colours to be used to colour $V(G)$ in $G+H$ are extended to set $\mathcal{C}'$, it follows that $b_k(G+H) \leq b_{k}(G) + b_k(H) +\underbrace{ \min(\sum\limits_{i=1}^{k}\theta_G(c_i)\theta_H(c_i)).}_{(over~all~ \delta^{(k)}-colourings~of~G~and~H~respectively.)}$ This inequality is illustrated with $K_3+K_3$.

First for equality and then for the relaxed colour set $\mathcal{C}' = \{c_1,c_2,c_3,c_4,c_5\}$. Let the vertices be $v_1,v_2,v_3$ and $u_1,u_2,u_3$ respectively. For the $\delta^{(2)}$-colouring $c(v_1)=c_1$, $c(v_2)=c_1$, $c(v_3)=c_2$ and $c(u_1)=c_2$, $c(u_2)=c_1$, $c(u_3)=c_2$ we have, $b_2(K_3+K_3)=1+1+(2\times 1)+(1\times 2) =6$. For $c(v_1)=c_2$, $c(v_2)=c_2$, $c(v_3)=c_1$ and $c(u_1)=c_2$, $c(u_2)=c_1$, $c(u_3)=c_2$ we have, $b_2(K_3+K_3)=1+1+(1\times 1)+(2\times 2) =7$. The minimum number of bad edges holds.

For the colouring $c(v_1)=c_1$, $c(v_2)=c_2$, $c(v_3)=c_3$ and $c(u_1)=c_4$, $c(u_2)=c_5$, $c(u_3)=c_5$ we have, $b_5(K_3+K_3)= 0+1+ 0 < 1+1+0$.

For graphs $G$ and $H$ it is known that for the corona operation:

\begin{equation*} 
\chi(G\circ H) =
\begin{cases}
\chi(G)+1, & \mbox{if $\chi(G)=\chi(H)$},\\
\chi(G), & \mbox{if $\chi(G)> \chi(H)$},\\
\chi(H)+1, & \mbox{if $\chi(G)< \chi(H)$}.
\end{cases}
\end{equation*} 

From the above different ranges for $1\leq k < \chi(G\circ H)$ are possible.  For some $\delta^{(k)}$-colouring of graph $H$ there will always exists a colour class in respect of say, $c_i$ such that $\theta^*_H(c_i)$ is an absolute minimum over all $\delta^{(k)}$-colourings of $H$ over all, $\theta_H(c_j)$. Utilising the result of Theorem \ref{Thm-2.7}, the result for the corona operation can be derived.

\begin{theorem}\label{Thm-2.8}
For graphs $G$ of order $n$, $V(G)= \{v_i:1\leq i \leq n\}$ and $n$ copies $H_i$, $1\leq i \leq n$ and $\mathcal{C} = \{c_i:1\leq i \leq \chi(G)-1\}$, $t\leq |\mathcal{C}|$ and $\mathcal{C}'=\{c_j:1\leq j \leq \chi(G\circ H)-1\}$, $k\leq |\mathcal{C}'|$ and $t=k$ if and only if $k\leq |\mathcal{C}|$, it follows that $b_k(G\circ H) = b_t(G) + b_k(H) + \sum\limits_{i=1}^{n}\theta^*_H(c_j)_{c(v_i)=c_j}$.
\end{theorem}
\begin{proof}
The result is a direct consequence of Theorem \ref{Thm-2.7}.
\end{proof}

For the colouring of $V(G)$ the relaxed colour set $\mathcal{C}'$ can be used to obtain, $b_k(G\circ H) = b_k(G) + b_k(H) + \sum\limits_{i=1}^{n}\theta^*_H(c_j)_{c(v_i)=c_j}$.

\section{Conclusion}

In \cite{3}, the $k$-defect polynomial of a graph $G$ is defined as the polynomial which determines the number of $\lambda$-colourings which result in $k$ bad edges. Clearly, the notion of $\delta^{(k)}$-colouring is a derivative of $k$-defect colouring. The critical difference is that for a $\delta^{(k)}$-colouring the number of colours, denoted by $k$, is prescribed and the condition of a minimum number of bad edges resulting from the $k$-colouring is set. In the $k$-defect polynomial, the number of colours denoted by $\lambda$ and the number of bad edges, denoted by $k$, are prescribed. The notational interchange is easy to do. Theorem \ref{Thm-2.5} and Theorem \ref{Thm-2.6} provide the first such relations because the number of bad edges is implied by the fact that the induced subgraph of any subset of vertices of a complete graph is itself, complete. Studying the relation between $\delta^{(k)}$-colouring and $k$-defect polynomials for graphs in general remains open. Graphs with large $\delta(G)$ motivates the introduction and the study of the notion of the $\delta^{(k)}$-polynomial of a graph.

It can easily be observed that if we permit more than one colour class to have adjacency, then the number of bad edges can be reduced. This observation applies well to cliques in graphs. In fact, if the upper bound is relaxed to be $1\leq k \leq \chi(G)$, then for $k=\chi(G)$, it is implied that $b_k(G) =0$. It is suggested that this derivative $\delta^{(k)}$-colouring is a worthy research avenue.

Finding an algorithm to determine the minimum term in the result of Theorem \ref{Thm-2.7} remains open as well. Determining $\theta^*_H(c_i)$ required in Theorem \ref{Thm-2.8} remains open as well.


\begin{thebibliography}{99}
\bibitem{1} J.A. Bondy and  U.S.R. Murty, (1976). \textit{Graph theory}, Springer, New York.

\bibitem{2} F. Harary, (2001). \textit{Graph theory}, Narosa Publ. House, New Delhi. 

\bibitem{3} E. Mphako-Banda, An introduction to the $k$-defect polynomials, \textit{Quaestiones Mathematicae}, (2018), 1-10.

\bibitem{4} D.B. West, (2001). \textit{Introduction to graph theory}, Prentice-Hall of India, New Delhi.

\end{thebibliography}
\end{document}